\documentclass[oneside]{amsart}

\usepackage{amsmath, amsthm, amsfonts}
\usepackage{amssymb}
\usepackage[utf8]{inputenc}
\usepackage{latexsym}
\usepackage{verbatim}
\usepackage{url}
\usepackage{textcomp}
\usepackage[all,cmtip]{xy}
\usepackage{color}
\usepackage{hyperref}
\usepackage{tabularx}
\input{xypic}
\usepackage{enumerate}
\usepackage{amssymb}
\usepackage{array}

\DeclareMathAlphabet{\mathfr}{U}{euf}{m}{n}

\newtheorem{theorem}{Theorem}[section]

\newtheorem{proposition}[theorem]{Proposition}
\newtheorem{corollary}[theorem]{Corollary}
\newtheorem{hypothesis}[theorem]{Hypothesis}
\newtheorem{lemma}[theorem]{Lemma}
\newtheorem{question}[theorem]{Question}

\newtheorem{claim}[theorem]{Claim}

\theoremstyle{remark}
\newtheorem{remark}[theorem]{Remark}
\newtheorem{definition}[theorem]{Definition}

\newcommand\acc[2]{\ensuremath{{}^{#1}\hskip-0.3ex{#2}}}

\newcommand{\lra}{\longrightarrow}
\newcommand{\ra}{\rightarrow}
\renewcommand{\inf}{\mathrm{inf}}
\newcommand{\res}{\mathrm{res}}

\newcommand{\Q}{\mathbb Q}
\newcommand{\Qbar}{{\overline{\mathbb Q}}}

\newcommand{\Gal}{\mathrm{Gal}}

\newcommand{\Z}{\mathbb Z}

\newcommand{\C}{\mathbb C}
\newcommand{\GL}{\mathrm{GL}}
\newcommand{\M}{\mathrm{M }}

\newcommand{\End}{\operatorname{End}}
\newcommand{\Hom}{\operatorname{Hom}}

\newcommand{\Jac}{\operatorname{Jac}}

\newcommand{\Res}{\operatorname{Res}}
\newcommand{\Tr}{\operatorname{Tr}}

\newcommand{\cO}{\mathcal{O}}

\newcommand{\D}{\mathcal{D}}
\newcommand{\Nm}{\operatorname{N}}
\newcommand{\fq}{\mathfrak{q}}
\newcommand{\fp}{\mathfrak{p}}
\newcommand{\cA}{\mathcal{A}}

\numberwithin{equation}{section}

\usepackage{url}

\newcommand{\II}{\mathbb{I}}

\newcommand{\cyc}[1]{{\mathrm{C}_#1}}
\newcommand{\dih}[1]{{\mathrm{D}_#1}}

\begin{document}
\title[Endomorphism algebras of $\Qbar$-split abelian surfaces over $\Q$]{Endomorphism algebras of geometrically split abelian surfaces over $\Q$}
\author{Francesc Fit\'e}
\address{Department of Mathematics\\
Massachussetts Institute of Technology \\
77 Massachussetts Avenue \\ 02139 Cambridge, Massachussetts\\ USA}
\email{ffite@mit.edu}
\urladdr{http://www-math.mit.edu/~ffite/}

\author{Xavier Guitart}
\address{Departament de Matemàtiques i Informàtica\\
Universitat de Barcelona\\Gran via de les Corts Catalanes, 585\\
08007 Barcelona, Catalonia
}
\email{xevi.guitart@gmail.com}
\urladdr{http://www.maia.ub.es/~guitart/}
\date{\today}

\begin{abstract} We determine the set of geometric endomorphism algebras of geometrically split abelian surfaces defined over $\Q$. In particular we find that this set has cardinality 92. The essential part of the classification consists in determining the set of quadratic imaginary fields $M$ with class group $\cyc 2 \times \cyc 2$ for which there exists an abelian surface $A$ defined over $\Q$ which is geometrically isogenous to the square of an elliptic curve with CM by $M$. We first study the interplay between the field of definition of the geometric endomorphisms of $A$ and the field $M$. This reduces the problem to the situation in which $E$ is a $\Q$-curve in the sense of Gross. We can then conclude our analysis by employing Nakamura's method to compute the endomorphism algebra of the restriction of scalars of a Gross $\Q$-curve.
\end{abstract}
\maketitle
\tableofcontents

\section{Introduction}

Let $A$ be an abelian variety of dimension $g\geq 1$ defined over a number field~$k$ of degree~$d$. Let us denote by $A_\Qbar$ its base change to $\Qbar$. We refer to $\End(A_\Qbar)$, the $\Q$-algebra spanned by the endomorphisms of~$A$ defined over $\Qbar$, as the $\Qbar$-endomorphism algebra of $A$. For a fixed choice of $g$ and $d$, it is conjectured that the set of possibilities for $\End(A_\Qbar)$ is finite. A slightly stronger form of this conjecture, applying to endomorphism rings of abelian varieties over number fields, has been attributed to Coleman in \cite{BFGR06}. 

Hereafter, let $A$ denote an abelian surface defined over $\Q$. In the case that~$A$ is geometrically simple (that is, $A_\Qbar$ is simple), the previous conjecture stands widely open. If $A$ is principally polarized and has CM it has been shown (see \cite{MU01}, \cite{BS17}, and \cite{KS23}) that $\End(A_\Qbar)$ is one of $13$ possible quartic CM fields. However, narrowing down to a finite set the possible quadratic real fields and quaternion division algebras over $\Q$ which occur as $\End(A_\Qbar)$ for some $A$ has escaped all attempts of proof. See also \cite{OS18} for recent more general results which prove Coleman's conjecture for CM abelian varieties. 

In the present paper, we focus on the case that $A$ is geometrically split, that is, the case in which~$A_\Qbar$ is isogenous to a product of elliptic curves, which we will assume from now on. Let~$\cA$ be the set of possibilities for $\End(A_\Qbar)$, where $A$ is a geometrically split abelian surface over $\Q$. 

Let us briefly recall how scattered results in the literature ensure the finiteness of~$\cA$ (we will detail the arguments in Section~\ref{section: main proofs}). Indeed, if $A_\Qbar$ is isogenous to the product of two non-isogenous elliptic curves, then the finiteness (and in fact the precise description) of the set of possibilities for $\End(A_\Qbar)$ follows from \cite[Proposition 4.5]{FKRS12}. If, on the contrary, $A_\Qbar$ is isogenous to the square of an elliptic curve, then the finiteness of the set of possibilities for $\End(A_\Qbar)$ was established by Shafarevich in \cite{Sha96} (see also \cite{Cre92} and \cite{Gon11} for the determination of the precise subset corresponding to modular abelian surfaces). In the present work, we aim at an effective version of Shafarevich's result. Our starting point is \cite[Theorem 1.4]{FG18}, which we recall in our particular setting.
\begin{theorem}[\cite{FG18}]\label{theorem: FG18}
If $A$ is an abelian surface defined over $\Q$ such that $A_\Qbar$ is isogenous to the square of an elliptic curve $E/\Qbar$ with complex multiplication (CM) by a quadratic imaginary field $M$, then the class group of $M$ is $1$, $\cyc 2$, or $\cyc 2 \times \cyc 2$.
\end{theorem}

It should be noted that several other works can be used to see that, in the situation of the theorem, the exponent of the class group of $M$ divides $2$ (see \cite{Sch07} or \cite{Kan11}, for example).

While it is an easy observation that an abelian surface $A$ as in the theorem can be found for each quadratic imaginary field $M$ with class group $1$ or $\cyc 2$ (see \cite[Remark~2.20]{FG18} and also Section~\ref{section: main proofs}), the question whether such an $A$ exists for each of the fields $M$ with class group $\cyc 2\times \cyc 2$ is far from trivial. The aforementioned results are thus not sufficient for the determination of the set $\cA$. The main contribution of this article is the following theorem.

\begin{theorem}\label{theorem: main}
Let $M$ be a quadratic imaginary field with class group $\cyc 2\times \cyc 2$. There exists an abelian surface defined over $\Q$ such that $A_\Qbar$ is isogenous to the square of an elliptic curve $E/\Qbar$ with CM by $M$ if and only if the discriminant of $M$ belongs to the set
\begin{align}\label{eq: set}
\{-84, -120, -132, -168, -228, -280,  -372, -408, -435,  -483, \\ \nonumber -520, -532, -595, -627, -708,  -795, -1012, -1435  \}.
\end{align}
\end{theorem}

The only imaginary quadratic fields with class group $\cyc 2\times \cyc 2$ whose discriminant does not belong to \eqref{eq: set} are
\begin{equation}\label{equation: fields}
\Q(\sqrt{-195}),\, \Q(\sqrt{-312}),\, \Q(\sqrt{-340}),\, \Q(\sqrt{-555}),\, \Q(\sqrt{-715}),\, \Q(\sqrt{-760}).
\end{equation}

With Theorem~\ref{theorem: main} at hand, the determination of the set $\cA$ follows as a mere corollary (see \S\ref{section: main proofs} for the proof).

\begin{corollary}\label{corollary: main}
The set $\cA$ of $\Qbar$-endomorphism algebras of geometrically split abelian surfaces over $\Q$ is made of:
\begin{enumerate}[i)]
\item $\Q\times \Q$, $\Q\times M$, $M_1\times M_2$, where $M$, $M_1$ and $M_2$ are quadratic imaginary fields of class number $1$;
\item $\M_2(\Q)$, $\M_2(M)$, where $M$ is a quadratic imaginary field with class group $1$, $\cyc 2$, or $\cyc 2 \times \cyc 2$ and distinct from those listed in \eqref{equation: fields}.
\end{enumerate}
In particular, the set $\cA$ has cardinality 92.
\end{corollary}

The paper is organized in the following manner. In Section~\ref{section: crepQcurve} we attach a $c$-representation $\varrho_V$ of degree $2$ to an abelian surface $A$ defined over $\Q$ such that $A_\Qbar$ is isogenous to the square of an elliptic curve $E/\Qbar$ with CM by $M$. It is well known that $E$ is a $\Q$-curve and that one can associate a $2$-cocycle~$c_E$ to~$E$. A $c$-representation is essentially a representation up to scalar and it is thus a notion closely related to that of projective representation. In the case of the $c$-representation $\varrho_V$ attached to $A$, the scalar that measures the failure of $\varrho_ V$ to be a proper representation is precisely the $2$-cocycle $c_E$. Choosing the language of $c$-representations instead of that of projective representations has an unexpected payoff: the tensor product of a $c$-representation $\varrho$ and its contragradient $c$-representation $\varrho^*$ is again a proper representation. We show that $\varrho_V \otimes \varrho_V^*$ coincides with the representation of $G_\Q$ on the 4 dimensional $M$-vector space $\End(A_\Qbar)$. This representation has been studied in detail in \cite{FS14} and the tensor decomposition of $\End(A_\Qbar)$ is exploited in Theorems~\ref{th: obstructions} and~\ref{theorem: dihedralmagic} to obtain obstructions on the existence of $A$. These obstructions extend to the general case those obtained in \cite[\S3.1,\S3.2]{FG18} under very restrictive hypotheses. The $c$-representation point of view also allows us to understand in a unified manner what we called \emph{group theoretic} and \emph{cohomological} obstructions in \cite{FG18}. It should be noted that one can define analogues of $\varrho_V$ in other more general situations. For example, a parallel construction in the context of geometrically isotypic abelian varieties potentially of $\GL_2$-type has been exploited in \cite{FG19} to determine a tensor factorization of their Tate modules. This can be used to deduce the validity of the Sato-Tate conjecture for them in certain cases.

In Section \ref{sec: restriction of scalars}, we describe a method of Nakamura to compute the endomorphism algebra of the restriction of scalars of certain Gross $\Q$-curves (see Definition \ref{def: Gross Q-curve} below for the precise definition of these curves). Then we apply this method to all Gross $\Q$-curves with CM by a field $M$ of class group $\cyc 2 \times \cyc 2$. This computation plays a key role in the proof of Theorem \ref{theorem: main}, both in proving the existence of the abelian surfaces for the fields $M$ different from those listed in \eqref{equation: fields}, and in proving the non-existence for the fields of \eqref{equation: fields}.

In Section \ref{section: main proofs} we culminate the proofs of Theorem~\ref{theorem: main} and Corollary~\ref{corollary: main} by assembling together the obstructions and existence results from Sections~\ref{section: crepQcurve} and~\ref{sec: restriction of scalars}. We essentially show that we can use the results of Section~\ref{section: crepQcurve} to reduce to the case of Gross $\Q$-curves, and then we deal with this case using the results of Section~\ref{sec: restriction of scalars}

\subsection*{Notations and terminology} For $k$ a number field, we will work in the category of abelian varieties up to isogeny over $k$.  Note that isogenies become invertible in this category. Given an abelian variety~$A$ defined over $k$, the set of endomorphisms $\End(A)$ of $A$ defined over~$k$ is endowed with a $\Q$-algebra structure. More generally, if $B$ is an abelian variety defined over $k$, we will denote by $\Hom(A,B)$ the $\Q$-vector space of homomorphisms from~$A$ to~$B$ that are defined over~$k$. We note that for us $\End(A)$ and $\Hom(A,B)$ denote what some other authors call $\End^0(A)$ and $\Hom^0(A,B)$. We will write $A\sim B$ to mean that~$A$ and~$B$ are isogenous over~$k$. If $L/k$ is a field extension, then $A_L$ will denote the base change of $A$ from $k$ to $L$. In particular, we will write $A_L\sim B_L$ if~$A$ and~$B$ become isogenous over~$L$, and we will write $\Hom(A_L,B_L)$ to refer to what some authors write as $\Hom_L(A,B)$.

\subsection*{Acknowledgements} Fit\'e is thankful to the organizers of the workshop ``Arithmetic Aspects of Explicit Moduli Problems" held at BIRS (Banff) in May 2017, where he explained Theorem~\ref{theorem: FG18} and raised the question on the existence of an abelian surface over $\Q$ with $\End(A_\Qbar)\simeq \M_2(M)$ for an $M$ with class group $\cyc 2 \times \cyc 2$. We thank Andrew Sutherland and John Voight for providing a positive answer to this question by pointing out the existence of an abelian surface (actually the Jacobian of a genus $2$ curve) with the desired property for the field $M=\Q(\sqrt{-132})$. We also thank Noam Elkies for providing three additional genus $2$ curves over $\Q$, these covering the fields $M=\Q(\sqrt{-408})$, $\Q(\sqrt{-435})$, and $\Q(\sqrt{-708})$. These four examples motivated the present paper. We are grateful to Marco Streng for alerting us of the existence of \cite{BS17} and \cite{KS23}.

Fit\'e was funded by the Excellence Program Mar\'ia de Maeztu MDM-2014-0445. Fit\'e was partially supported by MTM2015-63829-P. Guitart was funded by projects MTM2015-66716-P and MTM2015-63829-P. This project has received funding from the European Research Council (ERC) under the European Union's Horizon 2020 research and innovation programme (grant agreement No 682152).

\section{$c$-representations and $k$-curves}\label{section: crepQcurve} 

The goal of this section is to obtain obstructions to the existence of abelian surfaces defined over $\Q$ such that $\End(A_\Qbar)\simeq \M_2(M)$, where $M$ is a quadratic imaginary field. To this purpose, we analyze the interplay between the $k$-curves and $c$-representations that arise from them.

\subsection{$c$-representations: general definitions}
Let $V$ be a vector space of finite dimension over a field $k$ and let $G$ be a finite group. We say that a map
$$
\varrho_V:G\rightarrow \GL(V)
$$
is a $c$-representation (of the group $G$) if $\varrho_V(1)=1$ and there exists a map
$$
c_V:G \times G \rightarrow k^\times
$$
such that for every $\sigma,\tau \in G$ one has
\begin{equation}\label{equation: crep}
\varrho_V(\sigma)\varrho_V(\tau)=\varrho_V(\sigma\tau)c_V(\sigma,\tau)\,.
\end{equation}

\begin{remark}\label{remark: inverse}
The following properties follow easily from the definition:
\begin{enumerate}[i)]
\item Note that we have 
$$
\varrho_V(\sigma^{-1})=\varrho_V(\sigma)^{-1}c_V(\sigma^{-1},\sigma) \quad \text{and} \quad \varrho_V(\sigma^{-1})=\varrho_V(\sigma)^{-1}c_V(\sigma,\sigma^{-1})\,.
$$ 
In particular, $c_V(\sigma,\sigma^{-1})=c_V(\sigma^{-1},\sigma)$.
\item Note that if $c_V(\cdot,\cdot)=1$, the notion of $c$-representation corresponds to the usual notion of representation.
\end{enumerate}
\end{remark}

Let $V$ and $W$ be $c$-representations of the group $G$. Let $T=\Hom(V,W)$ denote the space of $k$-linear maps from $V$ to $W$. A homomorphism of $c$-representations from $V$ to $W$ is a $k$-linear map $f\in T$ such that 
$$
f(v)=\varrho_W(\sigma)(f(\varrho_V(\sigma)^{-1}v))
$$ 
for every $v\in V$ and $\sigma \in G$.

Consider now the map
$$
\varrho_T: G\rightarrow \GL(\Hom(V,W))\,,
$$
defined by 
$$
(\varrho_T(\sigma)f)(v)=\varrho_W(\sigma)(f(\varrho_V(\sigma)^{-1}v))\,.
$$

\begin{proposition}\label{equation: homcrep}
The map $\varrho_T$ together with the map
$c_T:G\times G \rightarrow k^\times$ defined by $c_T=c_V^{-1}\cdot c_W$ equip $T$ with the structure of a $c$-representation.
\end{proposition}

Before proving the proposition we show a particular case. In the case that $W$ is~$k$ equipped with the trivial action of $G$, let us write $V^*=T$ and $\varrho^*=\varrho_T$. In this case, $\varrho^*(\sigma)$ is the inverse transpose of $\varrho_V(\sigma)$. The assertion of the proposition is then immediate from \eqref{equation: crep}. 

The following two lemmas, whose proof is straightforward, imply the proposition.

\begin{lemma}
The maps 
$$
\varrho_{\otimes }:G \rightarrow \GL(V \otimes W)\,,
$$
defined by $\varrho_\otimes(\sigma)(v \otimes w)=\varrho_V(\sigma)(v) \otimes \varrho_W(\sigma)(w) $ and $c_{\otimes}=c_V\cdot c_W$ endow $V\otimes W$ with a structure of $c$-representation.
\end{lemma}

\begin{lemma}
The map
$$
\phi: W\otimes V^* \rightarrow T
$$
defined by $\phi(w\otimes f)(v)= f(v)w$ is an isomorphism of $c$-representations.

\end{lemma}

\begin{corollary}\label{equation: creprep}
When $V=W$, the $c$-representation $T$ is in fact a representation.
\end{corollary}

\subsection{$k$-curves: general definitions}\label{section: kcurvegendef}
We briefly recall some definitions and results regarding $\Q$-curves and, more generally, $k$-curves with complex multiplication. More details can be found in \cite[\S2.1]{FG18} and the references therein (especially \cite{Quer}, \cite{RiAVQ}, and \cite{Nak04}).

Let $E/\Qbar$ be an elliptic curve and let $k$ be a number field, whose absolute Galois group we denote by $G_k$.
\begin{definition}
  We say that $E$ is a \emph{$k$-curve} if for every $\sigma\in G_k$ there exists an isogeny $\mu_\sigma: \acc\sigma E \ra E$.
\end{definition}

\begin{definition}
  We say that $E$ is a \emph{Ribet $k$-curve} if $E$ is a $k$-curve and the isogenies $\mu_\sigma$ can be taken to be compatible with the endomorphisms of $E$, in the sense that the following diagram
\begin{align}\label{eq: compatibility}
  \xymatrix{
    \acc\sigma E \ar[d]^{\acc\sigma \varphi} \ar[r]^{\mu_\sigma} &E \ar[d]^\varphi\\
   \acc\sigma E \ar[r]^{\mu_\sigma}
&E}
\end{align}
commutes for all $\sigma\in G_k$ and all $\varphi\in \End(E)$. 
\end{definition}

\begin{remark}
\begin{enumerate}[i)]
\item Observe that if $E$ does not have CM, then $E$ is a $k$-curve if and only if it is a Ribet $k$-curve. If $E$ has CM (say by a quadratic imaginary field~$M$), it is well known that $E$ is isogenous to all of its Galois conjugates and hence it is always a $k$-curve; it is a Ribet $k$-curve if and only if $M\subseteq k$ (cf. \cite[Theorem 2.2]{silverman}).
\item We warn the reader that in the present paper we are using a slightly different terminology from that of \cite{FG18}: as in \cite{FG18} the only relevant notion was that of a Ribet $k$-curve, we called Ribet $k$-curves simply $k$-curves.
\end{enumerate}
\end{remark}
 Let $K$ be a number field containing $k$. We say that an elliptic curve $E/K$ is a $k$-curve \emph{defined over $K$} (resp. a Ribet $k$-curve \emph{defined over $K$}) if $E_\Qbar$ is a $k$-curve (resp. a Ribet $k$-curve). We will say that $E$ is \emph{completely defined over~$K$} if, in addition, all the isogenies $\mu_\sigma\colon \acc\sigma E \ra E$ can be taken to be defined over~$K$.

\begin{definition}\label{def: Gross Q-curve}
Let $H$ denote the Hilbert class field of $M$ and let $E/H$ be an elliptic curve with CM by $M$. We say that $E$ is a \emph{Gross $\Q$-curve} if $E$ is completely defined over $H$.
\end{definition}

The next proposition characterizes the existence of Gross $\Q$-curves and Ribet $M$-curves with CM by $M$ defined over the Hilbert class field $H$.

\begin{proposition}\label{proposition: McurveGrossCurve}
Let $M$ be a quadratic imaginary field and let $D$ denote its discriminant. Then:
\begin{enumerate}[i)]
\item There exists a Ribet $M$-curve $E^*$ with CM by $M$ and completely defined over~$H$.
\item There exists a Gross $\Q$-curve $E^*$ with CM by $M$ (and completely defined over $H$) if and only if $D$ is not of the form
\begin{equation}\label{equation: exceptional}
D=-4p_1 \dots p_{t-1}\,,
\end{equation}
where $t\geq 2$ and $p_1,\dots, p_{t-1}$ are primes congruent to $1$ modulo 4.
\end{enumerate}
\end{proposition}

The first part of the previous proposition is a weaker form of \cite[Proposition 5, p. 521]{Sh71} (see also \cite[Remark 1]{Nak01}). For the second part, we refer to \cite[\S11]{Gro80} and \cite[Proposition 5]{Nak04}.  Discriminants of the form \eqref{equation: exceptional} are called \emph{exceptional}.

Suppose from now on that $E$ is a $k$-curve defined over $K$ with CM by an imaginary quadratic field $M$. Fix a system of isogenies $\{\mu_\sigma:\acc\sigma E \ra E\}_{\sigma\in G_k}$. By enlarging~$K$ if necessary, we can always assume that $K/k$ is Galois and that~$E$ is completely defined over~$K$. We will equip $\End(E)$ with the following action. For $\sigma \in \Gal(K/k)$ and $\varphi \in \End(E)$ define 
$$
\sigma \star \varphi=\mu_{\sigma} \circ {}^{\sigma}\varphi \circ \mu_{\sigma}^{-1}\,.
$$
Note that if $E$ is a Ribet $k$-curve, then this action is trivial.
If we regard $M$ as a $\Gal(K/k)$-module by means of the natural Galois action (which is actually the trivial action when $k$ contains $M$) and $\End(E)$ endowed with the action defined above, then the identification of $\End(E)$ with $M$ becomes a $\Gal(K/k)$-equivariant isomorphism. The map 
\begin{align*}
\begin{array}{cccc}
  c_E^K\colon & \Gal(K/k)\times \Gal(K/k)&\lra & M^\times \\
  & (\sigma,\tau)&\longmapsto & \mu_{\sigma\tau}\circ \acc\sigma\mu_\tau^{-1}\circ \mu_\sigma^{-1}
\end{array}
\end{align*}
satisfies the condition
\begin{equation}\label{twoccocycle}
(\varrho \star c_E^K(\sigma, \tau))\cdot  c_E^K(\varrho\sigma,\tau)^{-1}\cdot  c_E^K(\varrho,\sigma \tau) \cdot c_E^K(\varrho,\sigma)^{-1}= 1,
\end{equation}
for $\varrho,\sigma,\tau \in \Gal(K/k)$, and is then a $2$-cocycle\footnote{Actually, this is the inverse of the cocycle considered in \cite{FG18}, but this does not affect any of the results that we will use.}. Denote by $\gamma_E^K$  the cohomology class in $H^2(\Gal(K/k),M^\times)$ corresponding to $c_E^K$. The class $\gamma_E^K$ only depends on the $K$-isogeny class of $E$. 


The next result is a consequence of Weil's descent criterion, extended to varieties up to isogeny by Ribet (\cite[\S8]{RiAVQ}).
\begin{theorem}[Ribet--Weil]\label{theorem: Weilsdescend} 
 Suppose that $E$ is a Ribet $k$-curve completely defined over $K$ (and hence $M\subseteq k$). Let $L$ be a number field with $k\subseteq L \subseteq K$, and consider the restriction map
\begin{align*}
  \res\colon H^2(\Gal(K/k),M^\times)\lra H^2(\Gal(K/L),M^\times).
\end{align*}
If $\res(\gamma_E^K)= 1$, there exists an elliptic curve $C/L$ such that $E\sim C_K$.
\end{theorem}

\subsection{$M$-curves from squares of CM elliptic curves}\label{section: QcurveCMsquare}

Let $M$ be a quadratic imaginary field. Let $A$ be an abelian surface defined over~$\Q$ such that $A_\Qbar$ is isogenous to $E^2$, where $E$ is an elliptic curve defined over $\Qbar$ with CM by~$M$. Let $K/\Q$ denote the minimal extension over which
$$
\End(A_\Qbar)\simeq \End(A_K)\,.
$$
By the theory of complex multiplication, $K$ contains the Hilbert class field~$H$ of~$M$. Note also that $K/\Q$ is Galois and the possibilities for $\Gal(K/\Q)$ can be read from \cite[Table 8]{FKRS12}. For our purposes, it is enough to recall that
\begin{equation}\label{equation: KoverM}
\Gal(K/M)\simeq \begin{cases}
\cyc{r} & \text{for }r\in \{1,2,3,4,6\},\\
\dih{r} & \text{for }r\in \{2,3,4,6\},\\
 A_4, S_4\,.
 \end{cases}
\end{equation}
Here, $\cyc r$ denotes the cyclic group of $r$ elements, $\dih r$ denotes the dihedral group of $2r$ elements, and $A_4$ (resp. $S_4$) stands for the alternating (resp. symmetric) group on $4$ letters.

We can (and do) assume that $E$ is in fact defined over $K$, and then we have that $A_K\sim E^2$. For $\sigma\in \Gal(K/\Q)$ we have that $({}^\sigma E)^2\sim {}^\sigma A_K=A_K \sim E^2 $. Therefore, Poincar\'e's decomposition theorem implies that $E$ is a $\Q$-curve completely defined over $K$. 

For the purposes of this article, we need to consider the following (slightly more general) situation: Let $N/M$ be a Galois subextension of $K/M$, and let~$E^*$ be a Ribet $M$-curve which is completely defined over~$N$ and such that $E_\Qbar\sim E^*_\Qbar$. Observe that there always exist $N$ and $E^*$ satisfying these conditions, for example by taking $N=K$ and $E^*=E$; but in \S\ref{section: crepCMsq} and \S\ref{section: obstructions}  below we will exploit certain situations where $N\subsetneq K$ and $E^*\neq E$.

Then we can consider two cohomology classes: the class $\gamma_E^K$ attached to $E$, and the class $\gamma_{E^*}^N$ attached to $E^*$. We recall the following key result about $\gamma_E^K$, which is a particular case of \cite[Corollary 2.4]{FG18}.
\begin{theorem}\label{th: 2-torsion}
The cohomology class $\gamma_E^K$ is $2$-torsion.  
\end{theorem}
Denote by $U$ the set of roots of unity of $M$ and put $P = M^\times/U$. The same argument of \cite[Proof of Theorem 2.14]{FG18} proves the following decomposition of the $2$-torsion of $H^2(\Gal(K/M),M^\times)$:
\begin{align}
  H^2(\Gal(K/M),M^\times)[2]\simeq H^2(\Gal(K/M),U)[2]\times \Hom(\Gal(K/M),P/P^2)\,.
\end{align}
If $M\neq \Q(i),\Q(\sqrt{-3})$ this particularizes to
\begin{align}\label{eq: dec of 2-torsion}
  H^2(\Gal(K/M),M^\times)[2]\simeq H^2(\Gal(K/M),\{\pm 1\})\times \Hom(\Gal(K/M),P/P^2).
\end{align}
For $\gamma\in  H^2(\Gal(K/M),M^\times)[2]$ we will denote by $(\gamma_\pm,\bar\gamma)$ its components under the isomorphism  \eqref{eq: dec of 2-torsion}; we will refer to $\gamma_\pm$ as the sign component and to $\bar\gamma$ as the degree component.

In order to study the relation between $\gamma_E^K$ and $\gamma_{E^*}^N$, define $L/K$ to be the smallest extension such that $E^*_L$ and $E_L$ are isogenous. Since all the endomorphisms of $E$ are defined over $K$, this is also the smallest extension $L/K$ such that $\Hom(E^*_L,E_L)=\Hom(E^*_\Qbar,E_\Qbar)$. The extension $L/\Q$ is Galois. Indeed, for $\sigma\in G_\Q$  put $L'=\acc\sigma L$ and let $\beta_\sigma\colon \acc\sigma E^* \ra E^*$ and $\mu_\sigma\colon \acc\sigma E\ra E$ be isogenies defined over $N$ and over $K$ respectively; then, if $\phi\colon E^*_L\ra E_L$ is an isogeny defined over $L$ we find that $\mu_\sigma\circ \acc\sigma\phi\circ \beta_\sigma^{-1}$ is an isogeny defined over $L'$ between $E^*_{L'}$ and $E_{L'}$, so that $L\subseteq L'$ and therefore $L=L'.$ 

One can also characterize $L/K$ as the minimal extension such that
$$
\Hom(E^*_\Qbar,A_\Qbar)\simeq \Hom(E^*_L,A_L)\,.
$$

Denote by
\begin{align*}
  \inf_N^K: H^2(\Gal(N/M), M^\times)\lra H^2(\Gal(K/M),M^\times) 
\end{align*}
 the inflation map in Galois cohomology. 
\begin{lemma}\label{lemma: inflations}
Suppose that $M\neq \Q(i),\Q(\sqrt{-3})$. Then
\begin{align*}
\inf_N^K(\gamma_{E^*}^N) = w\cdot\gamma_E^K,  
\end{align*}
for some $w\in H^2(\Gal(K/M),\{\pm 1\})$.
\end{lemma}
\begin{proof}
  Since $E_L\sim (E_*)_L$ we have that 
  \begin{align}\label{eq: inf}
\inf_N^L(\gamma_{E^*}^N)=\inf_K^L(\gamma_E^K).    
  \end{align}
Now consider the following piece of the inflation--restriction exact sequence
  \begin{align}\label{eq: inf res}
    H^1(\Gal(L/K),M^\times)\stackrel{t}{\longrightarrow} H^2(\Gal(K/M), M^\times)\xrightarrow{\inf_K^L} H^2(\Gal(L/M),M^\times).
  \end{align}
Equality \eqref{eq: inf} implies that $\inf_N^K(\gamma_{E^*}^N)$ and $\gamma_E^K$ have the same image under the inflation map $\inf_K^L$, and therefore $$\inf_N^K(\gamma_{E^*}^N)=t(v)\cdot \gamma_E^K$$ for some $v\in H^1(\Gal(L/K),M^\times)$. If $M\neq \Q(i),\Q(\sqrt{-3})$ we have that 
\begin{align*}
H^1(\Gal(L/K),M^\times) \simeq \Hom(\Gal(L/K),\{\pm 1\})  
\end{align*}
and therefore $t(v)$ belongs to $H^2(\Gal(K/M),\{\pm 1\})$.
\end{proof}

Observe that from Theorem \ref{th: 2-torsion} one cannot deduce that the class $\gamma_{E^*}^N$ is $2$-torsion, since $A_N$ is not isogenous to $(E^*)^2$ in general. By Lemma \ref{lemma: inflations}, what we do deduce is that $\inf_N^K(\gamma_{E^*}^N)^2 = 1$. Therefore, once again by the inflation--restriction exact sequence
  \begin{align}\label{eq: inf res}
    H^1(\Gal(K/N),M^\times)\stackrel{t}{\lra} H^2(\Gal(N/M), M^\times)\xrightarrow{\inf_N^K}H^2(\Gal(K/M),M^\times)
  \end{align}
we have that
\begin{align}
  \label{eq: transgression}
 (\gamma_{E^*}^N)^2=t(\mu) \text{ for some } \mu \in H^1(\Gal(K/N),M^\times).
\end{align}
The following technical lemma will be used in \S\ref{section: obstructions} below.
\begin{lemma}\label{lemma: pm1}
  Suppose that $N/M$ is abelian and that $M\neq \Q(i),\Q(\sqrt{-3})$. Let $c_{E^*}^N$ be a cocycle representing the class $\gamma_{E^*}^N$. Then $c_{E^*}^N(\sigma,\tau)=\pm c_{E^*}^N(\tau,\sigma)$ for all $\sigma,\tau\in \Gal(N/M)$.
\end{lemma}
\begin{proof}
Since $M\neq \Q(i),\Q(\sqrt{-3})$ we have that
\begin{align}
  \label{eq: h1hom}
  H^1(\Gal(K/N),M^\times)=\Hom(\Gal(K/N),\{\pm 1\}).
\end{align}

 By \eqref{eq: transgression} and \eqref{eq: h1hom} we can suppose that there exists a map $d:\Gal(N/M)\ra M^\times$ such that 
 \begin{align*}
c_{E^*}^N(\sigma,\tau)^2 = d(\sigma)d(\tau)d(\sigma\tau)^{-1}\cdot t(\mu)(\sigma,\tau),
 \end{align*}
where $t(\mu)(\sigma,\tau)\in \{\pm1\}$. Therefore
\begin{align*}
c_{E^*}^N(\sigma,\tau)^2=   \pm d(\sigma)d(\tau)d(\sigma\tau)^{-1} = \pm d(\sigma)d(\tau)d(\tau\sigma)^{-1} = \pm c_{E^*}^N(\tau,\sigma)^2. 
\end{align*}
We see that $\frac{c_{E^
*}^N(\sigma,\tau)}{c_{E^
*}^N(\tau,\sigma)}$ is a root of unity in $M$, hence $\pm 1$.
\end{proof}

\subsection{$c$-representations from squares of CM elliptic curves}\label{section: crepCMsq}

Keep the notations from Section~\ref{section: QcurveCMsquare}.
 We will denote by $V$ the $M$-module $\Hom(E^*_L,A_L)$. Fix a system of isogenies $\{\mu_\sigma:\acc\sigma E^* \ra E^*\}_{\sigma\in \Gal(L/M)}$.  We do not have a natural action of $\Gal(L/M)$ on $V$, but the next lemma says that we can use the chosen system of isogenies to define a $c$-action on $V$.

\begin{lemma}\label{lemma: Qcurvecact}  The map 
$$
\varrho_V:\Gal(L/M)\rightarrow \GL(V)
$$
defined by
$$
\varrho_V(f)={}^\sigma f\circ \mu_{\sigma}^{-1} \qquad \text{for } \sigma\in \Gal(L/M),\, f\in V
$$
and the $2$-cocycle $c_{E^*}^L$ endow the module $V$ with a structure of a $c$-representation. 
 \end{lemma}

\begin{proof} This is tautological:
$$
\varrho_V(\sigma)\varrho_V(\tau)(f)={}^{\sigma\tau}f\circ{}^{\sigma}\mu_\tau^{-1}\circ \mu_{\sigma}^{-1}={}^{\sigma\tau}f\circ\mu_{\sigma\tau}^{-1}\cdot c_{E^*}^L(\sigma,\tau)=\varrho_V(\sigma\tau)(f)c_{E^*}^L(\sigma,\tau)\,.
$$
\end{proof}

Let now $R$ denote the $M$-module $\End(A_K)$. It is equipped with the natural Galois conjugation action of $\Gal(L/M)$, which factors through $\Gal(K/M)$ and which we sometimes will write as $\varrho_R(\sigma)(\psi)={}^{\sigma}\psi$. Let $T$ denote $\Hom(V,V)$, equipped with the $c$-representation structure given by Lemma~\ref{lemma: Qcurvecact} and Proposition~\ref{equation: homcrep}. Note that by Corollary \ref{equation: creprep}, we know that $T$ is actually a $M[\Gal(L/M)]$-module.

\begin{lemma}\label{lemma: isoend}
The map
$$
\Phi: R\rightarrow T\simeq V\otimes V^*\qquad \Phi(\psi)(f)= \psi\circ f, \text{for }f\in V, \psi\in \End(A_K)
$$
is an isomorphism of $c$-representations (and thus of $M[\Gal(L/M)]$-modules). 
\end{lemma}

\begin{proof} The fact that $\Phi$ is a morphism of $c$-representations is  straightforward:
$$
\begin{array}{lll}
\varrho_T(\sigma)(\Phi(^{\sigma^{-1}}\psi))(f) & = &\varrho_V(\sigma)(\Phi(^{\sigma^{-1}}\psi)(\varrho_V(\sigma)^{-1}(f)))\,,\\[6pt]
 & = &\varrho_V(\sigma)(^{\sigma^{-1}}\psi\circ \varrho_V(\sigma^{-1})(f)c_{E^*}^L(\sigma^{-1},\sigma)^{-1})\,,\\[6pt]
 & = & \psi\circ f\circ{}^{\sigma}\mu_{\sigma^{-1}}^{-1}\mu_{\sigma}^{-1}c_{E^*}^L(\sigma^{-1},\sigma)^{-1}\,,\\[6pt]
 & = & \Phi(\psi)(f)\,,
\end{array}
$$
where we have used Remark~\ref{remark: inverse} in the second and last equalities. The lemma follows by noting that $\Phi$ is clearly injective and that both $R$ and $T$ have dimension~$4$ over~$M$.
\end{proof}

We now describe the $M[\Gal(K/M)]$-module structure of $R$. It follows from \eqref{equation: KoverM} that the order $r$ of an element $\sigma\in \Gal(K/M)$ is 1, 2, 3, 4, or 6.

\begin{lemma}\label{lemma: rhoR}
 $\Tr\varrho_R(\sigma)=2+ \zeta_r + \overline \zeta_r$, where $\zeta_r$ is a primitive $r$-th root of unity. 
\end{lemma}
\begin{remark}
  Note that this lemma is proven in \cite[Proposition 3.4]{FS14} under the strong running hypothesis of that paper: in our setting that hypothesis would say that there exists $E^*$ defined over $M$ such that $A_\Qbar\sim E_\Qbar^{*2}$ (i.e., that $N$ can be taken to be $M$, in the notation of the previous section).
\end{remark}
\begin{proof}
We claim that $\Tr(\varrho_R)\in M$ is in fact rational. Let us postpone the proof of this claim until the end of the proof of the lemma. Assuming it, we have that
\begin{equation}\label{equation: trrhoQ1}
\Tr_{M/\Q}(\Tr(\varrho_R(\sigma)))=2\Tr(\varrho_{R})(\sigma)\,.
\end{equation}
But if $\varrho_{R_\Q}$ is the representation afforded by $R$ regarded as an 8 dimensional module over $\Q$, we have
\begin{equation}\label{equation: trrhoQ}
\Tr_{M/\Q}(\Tr(\varrho_R(\sigma)))=\Tr(\varrho_{R_\Q})(\sigma)= 2(2+ \zeta_r+ \overline \zeta_r),
\end{equation}
where the last equality is \cite[Proposition 4.9]{FKRS12}. The comparison of \eqref{equation: trrhoQ1} and \eqref{equation: trrhoQ} concludes the proof of the lemma.

We turn now to prove the rationality of $\Tr\varrho_R$. We first recall the aforementioned proof (that of \cite[Proposition 3.4]{FS14}) which uses the fact that we can choose~$E^*$ to be defined over $M$. In this case, we have that $V$ is an $M[\Gal(L/M)]$-module, that $\Tr(\varrho_{V^*})$ is a sum of roots of unity so that $\Tr(\varrho_{V^*})= \overline {\Tr(\varrho_V)}$, and hence that $\Tr(\varrho_R)=\Tr(\varrho_V)\cdot \overline {\Tr\varrho_V}$ belongs to $\Q$. 

For the general case, assume that $\Tr\varrho_R$ does not belong to $\Q$. Since it is a sum of roots of unity of orders diving either 4 or 6, then $M$ would be $\Q(i)$ or $\Q(\sqrt{-3})$, but then we could take a model of $E^*$ defined over $M$, and by the above paragraph, the trace $\Tr\varrho_R$ would be rational, which is a contradiction.    
\end{proof}

\subsection{Obstructions}\label{section: obstructions}

Keep the notations from Section \ref{section: crepCMsq} and Section~\ref{section: QcurveCMsquare}. Let $S$ denote the normal subgroup of $\Gal(K/M)$ generated by the square elements.
In this section, we make the following hypotheses.

\begin{hypothesis} \label{hypothesis: feble}
\begin{enumerate}[i)]
\item There exists a Ribet $M$-curve $E^*$ with CM by $M$ completely defined over~$N$, where $N/M$ is the subextension of $K/M$ fixed by~$S$.
\item $M\not = \Q(i)$, $\Q(\sqrt{-3})$.
\end{enumerate}
\end{hypothesis}

Let $\sigma\in\Gal(K/M)$ be an element of order $r \in\{ 4, 6\}$. 
Let 
\begin{equation}\label{equation: projmap}
\bar \cdot : \Gal(K/M) \rightarrow \Gal(N/M)\simeq \Gal(K/M)/S
\end{equation}
denote the natural projection map. Note that $\Gal(N/M)$ is a group of exponent dividing $2$. 

\begin{theorem}\label{th: obstructions}
Under Hypothesis~\ref{hypothesis: feble}, we have:
\begin{enumerate}[i)]
\item If $r=4$, then $ 2c_{E^*}^N(\bar \sigma,\bar \sigma)$ belongs to $ \pm (M^{\times })^2$.
\item If $r=6$, then $ 3c_{E^*}^N(\bar \sigma,\bar \sigma)$ belongs to $ \pm (M^{\times })^2$.
\end{enumerate}
\end{theorem}

\begin{proof}
First of all, note that $E^*$ is completely defined over $N$. Thus we can, and do, assume that $c_{E^*}^L$ is the inflation of $c_{E^*}^N$. Let $s \in \Gal(L/M)$ be a lift of $\sigma$. By part ii) of Hypothesis~\ref{hypothesis: feble}, we have that $[L:K]\leq 2$. Therefore, the order of $s$ divides $2r$. We then have
\begin{equation}\label{equation: prodmagic}
\varrho_V(s)^{2r}=\varrho_V( s ^2 )^rc_{E^*}^N(\bar \sigma,\bar \sigma)^r= \varrho_V(s^{2r}) c_{E^*}^N(\bar\sigma,\bar\sigma)^r=c_{E^*}^N(\bar\sigma,\bar\sigma)^r\,,
\end{equation}
where we have used that $c_{E^*}^N(\bar \sigma^{2e},\bar \sigma^{2e'})=1$ for any pair of integers $e,e'$. Let $\alpha$ and~$\beta$ be the eigenvalues of $\varrho_V( s)$. By \eqref{equation: prodmagic}, we have that $\alpha^{2r}= c_{E^*}^N(\bar\sigma,\bar\sigma)^r$, from which we deduce that $\omega_r\alpha^2=c_{E^*}^N(\bar\sigma,\bar\sigma)\in M^\times$, where $\omega_r$ is a  (not necessarily primitive) $r$-th  root of unity.  

Since the eigenvalues of $\varrho_{V^*}(s)$ are $1/\alpha$ and $1/\beta$, by Lemmas \ref{lemma: rhoR} and \ref{lemma: isoend} we have that
\begin{equation}\label{equation: rcyclotomic}
2+ \zeta_r + \overline \zeta_ r=(\alpha+\beta)\left(\frac{1}{\alpha } +\frac{1}{\beta}\right); \text{ equivalently, }   \alpha^2+\beta^2=(\zeta_r+\overline \zeta_r)\alpha \beta\,.
\end{equation}
This means that $\alpha/\beta$ satisfies the $r$-th cyclotomic polynomial and thus, by reordering $\alpha$ and $\beta$ if necessary, we have that $\alpha=\beta\zeta_r$.

Combining this with \eqref{equation: rcyclotomic}, we get
$$
(2+\zeta_r+ \overline \zeta_r)c_{E^*}^N(\bar \sigma,\bar \sigma)=(2+\zeta_r+ \overline \zeta_r)\omega_r\alpha^2=(2+\zeta_r+ \overline \zeta_r)\alpha \beta\omega_r\zeta_r=(\alpha+\beta)^2\omega_r\zeta_r\,.
$$
Since the left-hand side is in $M^\times$, the fact that $\alpha+\beta \in M^\times$ tells us that $\omega_r\zeta_r\in M^\times$. If $\omega_r\zeta_r$ is not rational, then $M=\Q(\zeta_r)$, which contradicts part ii) of Hypothesis~\ref{hypothesis: feble}. If $\omega_r\zeta_r\in\Q$, since it is a root of unity, it must be $\pm 1$ and thus we get
$$
\pm(2+\zeta_r+ \overline \zeta_r)c_{E^*}^N(\bar \sigma,\bar \sigma)=(\alpha+\beta)^2\, .
$$
Therefore, $(2+\zeta_r+ \overline \zeta_r)c_{E^*}^N(\bar \sigma,\bar \sigma)$ belongs to $ \pm (M^{\times})^2$.
\end{proof}

\begin{remark}\label{remark: LM4}
Note that it follows from the above proof that if $r=4$, then any lift $s \in \Gal(L/M)$ of $\sigma$ has order $2r=8$. Indeed, if the order of $s$ was $r$, then arguing as in \eqref{equation: prodmagic}, we would obtain $\varrho_V(s)^r=c_{E^*}^N(\bar \sigma, \bar \sigma)^{r/2}$, from which we would infer $\omega_{r/2}\alpha^2=c_{E^*}^N(\bar\sigma,\bar\sigma)$, for some (not necessarily primitive) $r/2$-th root of unity. We could then run the same argument as above, but since $\omega_{r/2}\zeta_r$ is never rational, we would deduce now that $M=\Q(i)$. Note that if $r=6$ it can certainly happen that $\omega_{r/2}\zeta_r\in\Q$.
\end{remark}

Until the end of this section, we make the following additional assumption on~$M$.
\begin{hypothesis}\label{hypothesis: forta} 
\begin{enumerate}[i)]
\item $\Gal(K/M)\simeq \dih 4$ or $\dih 6$.
\item $M\not = \Q(i)$, $\Q(\sqrt{-3})$.
\end{enumerate}
\end{hypothesis} 

Hypothesis i) implies that $N/M$ is a biquadratic extension. By part i) of Proposition~\ref{proposition: McurveGrossCurve}, there exists a Ribet $M$-curve $E^*$ with CM by $M$ completely defined over the Hilbert class field $H$ of $M$. Using \cite[Theorem 2.14]{FG18}, it is immediate to see that $H\subseteq N$, so that Hypothesis~\ref{hypothesis: forta} implies Hypothesis~\ref{hypothesis: feble}.

The next two propositions describe the structure of the group $\Gal(L/M)$.

\begin{proposition}\label{proposition: LM4}
 If $\Gal(K/M)\simeq \dih 4$, then $\Gal(L/M)$ is isomorphic to either the dihedral group $\dih 8$; the generalized dihedral group $\mathrm{QD}_8$ of order $16$; or the generalized quaternion group $\mathrm Q_{16}$\footnote{The gap identification numbers of $\mathrm{QD}_8$ and $\mathrm Q_{16}$ are $\langle 16,8\rangle$ and $\langle 16,9\rangle$, respectively.}.
\end{proposition}

\begin{proof}
If $\Gal(K/M)\simeq \dih 4$, then by Remark \ref{remark: LM4} we have that any element of $\Gal(L/M)$ projecting onto an element of $\Gal(K/M)$ of order $4$ must have order~8. The groups of order 16 with a quotient isomorphic to $\dih 4$ satisfying the previous property are those in the statement of the proposition. 
\end{proof}

\begin{proposition}\label{proposition: LM6}
If $\Gal(K/M)\simeq \dih 6$, there exists a Ribet $M$-curve $E^*$ completely defined over $N$ with CM by $M$ such that $E \sim E^*_K$ and hence $L=K$ and $\Gal(L/M)\simeq \dih 6$.
\end{proposition}
\begin{proof}
  Recall the cohomology class $\gamma_E^K\in H^2(\Gal(K/M),M^\times)[2]$ attached to $E$ and consider the restriction map
  $$
  \res: H^2(\Gal(K/M),M^\times)\rightarrow H^2(\Gal(K/N),M^\times)\,.
  $$ 
  We will first see that $\gamma=\res\gamma_E^K$ is trivial. Recall the decomposition \eqref{eq: dec of 2-torsion} of the $2$-torsion cohomology classes into degree and sign components
  \begin{align*}
    H^2(\Gal(K/N),M^\times)[2]\simeq  H^2(\Gal(K/N),\{\pm 1\})\times \Hom(\Gal(K/N),P/P^2),
  \end{align*}
and the notation $\gamma_\pm$ (resp. $\bar\gamma$) for the sign component (resp. degree component) of $\gamma$. Since $\Gal(K/N)\simeq \cyc 3$ is the subgroup of $\Gal(K/M)$ generated by the squares, we have that $\bar\gamma$ is trivial. Since
$$
H^2(\Gal(K/N),\{\pm 1\})\simeq H^2(\cyc 3, \{\pm 1\})= 0\,,
$$
we see that $\gamma_{\pm}$ is also trivial. By Theorem \ref{theorem: Weilsdescend}, there exists an elliptic curve $E^*$ defined over~$N$ such that $E^*_K\sim E$. To see that $E^*$ is completely defined over $N$, on the one hand, note that since $M\not = \Q(i), \Q(\sqrt{-3})$, then $E^*$ and any Galois conjugate ${}^{\sigma}E^*$ of it are isogenous over a quadratic extension of $N$. On the other hand, since  $E^*_K\sim E$ and $E$ is completely defined over $K$, we have that the smallest field of definition of $\Hom(E_\Qbar^*,\acc\sigma E_\Qbar^*)$ is contained in $K$. Since $K/N$ is a cubic extension, we deduce that $E^*$ and ${}^\sigma E^*$ are in fact isogenous over $N$. 
\end{proof}

\begin{corollary}\label{corollary: elementsgroup} If $\Gal(K/M)\simeq \dih r$ for $r=4$ or $6$, there exists a Ribet $M$-curve $E^*$ with CM by $M$ completely defined over $N$ for which $\Gal(L/M)$ contains
\begin{enumerate}[i)]
\item an element $s$ of order $8$ if $r=4$ and of order $6$ if $r=6$;
\item an element $t$ such that $tst^{-1}=t^a$ for $1\leq a\leq 2r$ such that $a\equiv -1 \pmod r$.
\end{enumerate}
\end{corollary}
\begin{proof}
This is obvious when $\Gal(L/M)$ is dihedral. For the other options allowed by Proposition~\ref{proposition: LM4}, recall that
$$
\mathrm{QD}_8\simeq \langle s,t \,|\, s^8, t^2,tsts^5\rangle\,,\qquad \mathrm{Q}_{16}\simeq \langle s,t \,|\, s^8, t^2s^4,tst^{-1}s\rangle\,.
$$
\end{proof}

\begin{remark}\label{remark: ele}
It is clear from the proof of Proposition~\ref{proposition: LM6} that, in the case that $N=H$ and $H$ is not exceptional, we can choose $E^*$ in the above corollary to be a Gross $\Q$-curve.
\end{remark}

Until the end of this section, we will assume that $E^*$ is as in the previous corollary. Let $s$ and $t$ be also as in the corollary, and let $\sigma$ and $\tau$ be the images of $s$ and $t$ under the projection map 
$$
\Gal(L/M)\rightarrow \Gal(K/M)\,.
$$ 
Recall also the projection map $\bar \cdot : \Gal(K/M)\rightarrow \Gal(N/M)$ and note that $\bar \sigma$ and $\bar \tau$ are non-trivial elements of $\Gal(N/M)$.
\begin{theorem}\label{theorem: dihedralmagic}
Under Hypothesis~\ref{hypothesis: forta}, we have $ c_{E^*}^N(\bar \tau, \bar \tau)=\pm 1$.
\end{theorem}

\begin{proof}
By Lemma~\ref{lemma: pm1}, we have that $c_{E^*}^N(g,g')=\pm c_{E^*}^N(g',g)$ for every $g,g'\in \Gal(N/M)$. Moreover, the 2-cocycle condition \eqref{twoccocycle} asserts that
$$
c_{E^*}^N(\bar\tau,\bar\tau)=c_{E^*}^N(\bar\tau,\bar\tau)c_{E^*}^N(\bar\sigma, 1)=c_{E^*}^N(\bar\sigma\bar\tau, \bar\tau)c_{E^*}^N(\bar\sigma,\bar\tau)\,.
$$
Then, we have 
\begin{equation}\label{equation: dihedral rel}
\begin{array}{lll}
\varrho_V(t)\varrho_V(s)\varrho_V(t)^{-1} & = & \varrho_V(t)\varrho_V(s)\varrho_V(t^{-1})c_{E^*}^N(\bar \tau,\bar\tau)=\\[6pt]
 & = & \varrho_V(ts)\varrho_V(t^{-1})c_{E^*}^N(\bar\tau,\bar\sigma)c_{E^*}^N(\bar\tau,\bar\tau)=\\[6pt]
 & = & \varrho_V(tst^{-1})c_{E^*}^N(\bar \tau\bar\sigma,\bar\tau)c_{E^*}^N(\bar\tau,\bar\sigma)c_{E^*}^N(\bar\tau,\bar\tau)=\\[6pt]
 & = & \pm \varrho_V(s^{a})c_{E^*}^N(\bar\tau,\bar\tau)^2\,.
\end{array}
\end{equation}
It is easy to observe that 
\begin{equation}\label{equation: powervarrho}
\varrho_V(s)^a= \varrho_V(s^a)c_{E^*}^N(\bar\sigma,\bar\sigma)^{(a-1)/2}\,.
\end{equation} 
Letting $\alpha$ and $\beta$ be the eigenvalues of $\varrho_V(s)$, taking traces of \eqref{equation: dihedral rel}, and applying \eqref{equation: powervarrho}, we obtain
$$
(\alpha+\beta)=\pm \left(\alpha^a+\beta^a \right)c_{E^*}^N(\bar\sigma,\bar\sigma)^{-(a-1)/2}c_{E^*}^N(\bar\tau,\bar\tau)^2
$$
But as in the proof of Theorem \ref{th: obstructions}, we have $\beta=\zeta_r\alpha$ and $c_{E^*}^N(\bar \sigma, \bar \sigma)=\omega_r\alpha^2$, where $\zeta_r$ and $\omega_r$ are $r$-th roots of unity and $\zeta_r$ is  primitive. This, together with the fact that $a\equiv -1 \pmod r$, permits to write the above equation as
$$
\pm \frac{1+ \zeta_r}{\omega_r^{-(a-1)/2}(1+\overline \zeta_r)} = c_{E^*}^N(\bar\tau,\bar\tau)^{2}\in (M^{\times})^2.
$$
One easily verifies that $(1+\zeta_r)/(1+\overline \zeta_r)$ is an $r$-th root of unity. Therefore, the left-hand side of the above equation is a root of unity in $M^\times $, and hence it must be $\pm 1$. 
\end{proof}

\section{Restriction of scalars of Gross $\Q$-curves}\label{sec: restriction of scalars}
For the convenience of the reader, in this section we review some results of Nakamura \cite{Nak04} on Gross $\Q$-curves, to which we refer for more details and proofs. 

Let $M$ be an imaginary quadratic field. Throughout this section, we make the following hypothesis.

\begin{hypothesis}\label{hypothesis: non-exceptional} 
\begin{enumerate}[i)]
\item $M$ is non-exceptional.
\item $M$ has class group isomorphic to $\cyc 2 \times \cyc 2$.
\end{enumerate}
\end{hypothesis}

\begin{remark}\label{remark: discriminants}
If $M$ has class group isomorphic to $\cyc 2 \times \cyc 2$, then the discriminant~$D$ of $M$ belongs to the set
\begin{align*}
\{-84, -120, -132, -168, -195, -228, -280, -312, -340, -372, -408, -435, -483,\\  -520, -532, -555, -595, -627, -708, -715, -760, -795, -1012, -1435  \}.
\end{align*}
This list can be easily obtained from \cite{Wat03}, for example. Among them, only $-340$ is exceptional.
\end{remark}

Then, by Proposition~\ref{proposition: McurveGrossCurve}, there exists a Gross $\Q$-curve $E$ with CM by $M$, which is thus completely defined over the Hilbert class field $H$ of  $M$. The aim of the present section is to describe Nakamura's method for computing the endomorphism algebra of the restriction of scalars of a Gross $\Q$-curve, which we will then apply to all Gross $\Q$-curves attached to $M$ satisfying Hypothesis~\ref{hypothesis: non-exceptional}. Our account of Nakamura's method will be only in the particular case where $M$ has class group $\cyc 2 \times\cyc 2$, which is the case of interest to us.

As seen in Section \ref{section: kcurvegendef}, one can associate to $E$ a cohomology class $\gamma_E:=\gamma_E^H$ in the group $H^2(\Gal(H/\Q), M^\times)$. The set of cohomology classes arising from Gross $\Q$-curves over $H$ has cardinality $8$ (cf. \cite[Proposition 4]{Nak04}), and we regard the set of Gross $\Q$-curves over $H$ as partitioned into $8$ equivalence classes according to their cohomology class.

Let $\Res_{H/M}(E)$ denote Weil's restriction of scalars of $E$. This variety is a priori defined over $M$, but it can be defined over $\Q$, in the sense that $\Res_{H/M}(E) \simeq (B_E)_M$ for some variety $B_E/\Q$. It turns out that  the endomorphism algebra  $\D_E = \End(B_E)$ only depends on the cohomology class $\gamma_E$ \cite[Proposition 6]{Nak04}. Nakamura devised a method for computing $\D_E$ in terms of the Hecke character attached to $E$, which he applied to compute all the endomorphism algebras arising in this way from Gross $\Q$-curves in the cases where $D=-84$ and $D=-195$. We extend his computation to the remaining $21$ non-exceptional discriminants of Remark~\ref{remark: discriminants}.

\subsection{Hecke characters of Gross $\Q$-curves}
The first step is to compute a set of Hecke characters whose associated elliptic curves represent all the equivalence classes of Gross $\Q$-curves.

\subsubsection*{Local characters} We begin by defining certain local characters that will be used to describe the Hecke characters. Let $\II_M$ be the group of ideles of $M$. If $\fp$ is a prime of $M$, we denote by $U_{\fp}=\cO_{M,\fp}^\times$ the group of local units. Also, for a rational prime~$p$ put $U_p=\prod_{\fp\mid p}U_\fp $.

Suppose that $p$ is odd and inert in $M$. Then define $\eta_p$ as the unique character $\eta_p\colon U_p\ra \{\pm 1\}$ such that $\eta_p(-1)= (-1)^{\frac{p-1}{2}}$. 

Suppose now that $2$ is ramified in $M$ and write $D=4m$. If $m$ is odd, then 
\begin{align*}
U_2/U_2^2 \simeq \left(\Z/2\Z \right)^3\simeq  \langle \sqrt{m},3-2\sqrt{m},5 \rangle.
\end{align*}
Define $\eta_{-4}\colon U_2\ra \{\pm1\}$ to be the character with kernel $\langle 3-2\sqrt{m},5 \rangle$. If $m$ is even then 
\begin{align*}
U_2/U_2^2 \simeq \left(\Z/2\Z \right)^3 \simeq \langle 1+\sqrt{m},-1,5 \rangle.
\end{align*}
Define $\eta_{8}$ to be the character with kernel $\langle 1+\sqrt{m},-1 \rangle $ and $\eta_{-8}$ the character with kernel $\langle 1+\sqrt{m},-5 \rangle$.

\subsubsection*{Hecke characters}
Let $U_M =\prod_\fp U_{\fp}$ be the maximal compact subgroup of $\II_M$. Let~$S$ be a finite set of primes of $M$ and put $U_S =\prod_{\fp\in S}U_\fp$. Suppose that $\eta\colon U_S\ra \{\pm 1\}$ is a continuous homomorphism such that $\eta(-1)=-1$. Next, we explain how to construct from $\eta$ a Hecke character $\phi: \II_M\ra \C^\times$ (not uniquely determined) that gives rise, in certain cases, to a Gross $\Q$-curve.

First of all, extend $\eta$ to a character that we denote by the same name $\eta\colon U_M\ra \{\pm 1\}$ by composing with the projection $U_M\ra U_S$. Now this character $\eta$ can be extended to a character  $\tilde\eta\colon U_MM^\times M_\infty^\times\lra \C^\times$ by imposing that
\begin{align}\label{eq: extension of eta}
  \tilde \eta(M^\times) = 1, \ \ \tilde\eta(z) = z^{-1} \text{ for } z\in M_\infty^\times.
\end{align}
Let $\phi\colon \II_M\ra \C^\times$ be a Hecke character that extends $\tilde \eta$ (there are $[H:M]=4$ such extensions, cf. \cite[p. 523]{Sh71}). For future reference, it will be useful to have the following formula for $\phi$ evaluated at certain principal ideals.
\begin{lemma}
Suppose that $(\alpha)$ is a principal ideal of $M$ such that $v_\fp (\alpha)= 0$ for all $\fp\in S$, and denote by $\alpha_S\in U_S$ the natural image of $\alpha$ in $U_S$. Then
\begin{align}\label{eq: phi at alpha}
  \phi((\alpha)) = \eta(\alpha_S)\alpha_\infty,
\end{align}
where $\alpha_\infty$ denotes the image of $\alpha$ in $M_\infty = \C$.  
\end{lemma}
\begin{proof}
 If we write $(\alpha)=\prod_{\fq\in T}\fq^{v_\fq(\alpha)}$, where $T$ denotes the support of $(\alpha)$, then
\begin{align*}
  \phi((\alpha)) = \prod_{\fq\in T}\phi_\fq(\alpha_\fq),
\end{align*}
where $\phi_\fq$ denotes the restriction of $\phi$ to $M_\fq$ and $\alpha_\fq $ the image of $\alpha $ in $M_\fq$. Observe that by hypothesis $S\cap T=\emptyset$, and that if $\fq\not\in S\cup T$, then $\phi_\fq(\alpha_\fq)=1$, since $\alpha_\fq$ belongs to $U_\fq$ and $\phi_{|U_{\fq}}=\tilde \eta_{|U_\fq}=1$. Therefore, we can write
\begin{align*}
  \phi((\alpha)) = \prod_{\fq\in T}\phi_\fq(\alpha_\fq)\prod_{\fq\not\in T}\phi_\fq(\alpha_\fq)\prod_{\fq\in S}\phi_\fq^{-1}(\alpha_\fq) = \left(\prod_{\fq}\phi_\fq(\alpha_\fq)\right)\eta(\alpha_S),
\end{align*}
where we have used that $\eta$ has order $2$. Then, by \eqref{eq: extension of eta} we have that
\begin{align*}
  \phi((\alpha)) = \left(\phi_\infty (\alpha_\infty)\prod_{\fq}\phi_\fq(\alpha_\fq)\right)\phi_\infty(\alpha_\infty)^{-1}\eta(\alpha_S) = \phi(\alpha) \alpha_\infty \eta(\alpha_S)=\alpha_\infty \eta(\alpha_S).
\end{align*}
\end{proof}
Define now a Hecke character of $H$ by means of $\psi = \phi\circ \Nm_{H/M}$, where 
$$\Nm_{H/M}\colon \II_H\ra \II_M
$$ 
denotes the norm on ideles.  By a result of Shimura \cite[Proposition 9]{Sh71}, the Hecke character $\psi$ is attached to a Gross $\Q$-curve if and only if $\bar\phi = \phi$, where the bar denotes the action of complex conjugation.

 For example, if $D$ has some prime factor $q\equiv 3 \pmod 4$, put $\eta_0 = \eta_q$. If all the odd primes dividing $D$ are congruent to $1$ modulo $4$, then $D=8m$ for some odd $m$ and we define $\eta_0$ to be $\eta_{-8}$. If we denote by $\phi_0\colon \II_M\ra\C^\times$ a Hecke character attached to $\eta_0$ by the above construction, then the Hecke character $\psi_0=\phi_0\circ\Nm_{H/M}$ is the Hecke character attached to a Gross $\Q$-curve over $H$.

Let $W$ be the set of characters $\theta\colon U_M\ra \{\pm 1\}$ such that $\theta(-1)=1$ and $\bar\theta = \theta$. Denote also by $W_0$ the set of $\theta\in W$ such that $\theta = \kappa\circ \Nm_{M/\Q}$ for some Dirichlet character $\kappa$. By \cite[Proposition 3]{Nak04}, the group $W/W_0$ is generated by two characters that  can be described explicitly in terms of the characters $\eta_p, \eta_{-4},\eta_{-8}$, and $\eta_{8}$. More precisely:
\begin{enumerate}
\item If $D=-pqr$ with $p$, $q$, and $r$ primes congruent to $3$ modulo $4$, then $W/W_0 = \langle \eta_p\eta_q,\eta_p\eta_r \rangle $.
\item If $D=-pqr$ with $p$ and $q$ primes congruent to $1$ modulo $ 4$, and $r\equiv 3 \pmod 4$, then $W/W_0 = \langle \eta_p, \eta_q\rangle $.
\item If $D =-4pq$ with $p$ and $q$ congruent to $3
$ modulo $ 4$, then $W/W_0 = \langle \eta_{-4},\eta_p\eta_q\rangle $.
\item If $D =-8pq$ with $p$ and $q$ congruent to $3$ modulo $ 4$ then $W/W_0 = \langle \eta_{-8}\eta_p,\eta_{-8}\eta_q\rangle $.
\item If $D =-8pq$ with $p\equiv 1\pmod 4$ and $q\equiv 3\pmod 4$ then $W/W_0 = \langle \eta_{8},\eta_p\rangle $.
\item If $D = -8pq$ with $p$ and $q$ congruent to $1$ modulo $4$, then $W/W_0 = \langle \eta_{p},\eta_q\rangle $.
\
\end{enumerate}
Denote by $\tilde \omega_1,\tilde \omega_2$ the generators of $W/W_0$, and define $\omega_i=\tilde \omega_i\circ \Nm_{H/M}$.

Now let $k/H$ be a quadratic extension such that $k/\Q$ is Galois and $k/M$ is non-abelian. Such quadratic extensions exist by \cite[Theorem 1]{Nak04}. Denote by $\chi\colon \II_H\ra \{\pm 1\}$ the Hecke character attached to $k/H$.

By \cite[Theorem 2]{Nak04}, the eight equivalence classes of $\Q$-curves over $H$  are represented by the Hecke characters $\psi_0\cdot \omega$ with $\omega \in \langle \omega_1,\omega_2,\chi \rangle $. Observe that, in particular, this set of Hecke characters does not depend on the choice of $k$ (any $k$ which is Galois over $\Q$ and non-abelian over $M$ will produce the same set of Hecke characters).

\subsection{Method for computing the endomorphism algebra}

Let $\fp_1$ and $\fp_2$ be prime ideals of $M$ that generate the class group and that are coprime to the conductors of $\psi_0$, $\omega_1$, $\omega_2$, and $\chi$. Let $L_i$ be the decomposition field of $\fp_i$ in $H$, and $F_i$ the maximal totally real subfield of $L_i$. 

Suppose that $E$ is a Gross $\Q$-curve over $H$ with Hecke character of the form $\psi = \psi_0 \omega_1^a \omega_2^b$ for some $a,b\in\{0,1\}$. We can write $\psi=\phi\circ \Nm_{H/M}$, where $\phi=\phi_0\tilde \omega_1^a\tilde \omega_2^b$. Then $\phi(\fp_i) + \phi(\bar\fp_i)$ generates a quadratic number field  $\Q(\sqrt{n_i})$, and the endomorphism algebra  $\D_E=\End(B_E)$ is isomorphic to the biquadratic field $\Q(\sqrt{n_1},\sqrt{n_2})$ (cf. \cite[Proposition 7, Theorem 3]{Nak04}). 
\begin{remark}
  Observe that $\phi(\fp_i) + \phi(\bar\fp_i)$ can be computed if one knows the two quantities $\phi (\fp_i^2)$ and $\phi(\fp_i\bar\fp_i)$. Since $\fp_i^2$ and $\fp_i\bar\fp_i$ are principal, one can compute $\phi (\fp_i^2)$ and $\phi(\fp_i\bar\fp_i)$ by means of  \eqref{eq: phi at alpha}.
\end{remark}

Suppose now that the Hecke character of $E$ is of the form $\psi=\psi_0 \chi\omega_1^a  \omega_2^b$. Then $\D_E$ is a quaternion algebra over $\Q$, say $\D_E\simeq \left( \frac{t_1,t_2}{\Q}\right)$. The $t_i$ can be computed as follows (see \cite[Proposition 7]{Nak04}). First of all, let $n_1$ and $n_2$ be the rational numbers defined as in the previous paragraph for the character $\psi/\chi = \psi_0\omega_1^a\omega_2^b$.
\begin{enumerate}
\item Suppose that $\Gal(k/L_i)\simeq \cyc 2 \times \cyc 2$ . Then:
  \begin{enumerate}
  \item If $k/F_i$ is abelian then $t_i = n_i$.
\item If $k/F_i$ is non-abelian, then $t_i = D/n_i$.
  \end{enumerate}
\item Suppose that $\Gal(k/L_i)\simeq \cyc 4$. Then:
  \begin{enumerate}
  \item If $k/F_i$ is abelian, then $t_i=-n_i$.
\item If $k/F_i$ is non-abelian, then $t_i = -D/n_i$.
  \end{enumerate}
\end{enumerate}

\subsection{Computations and tables}
For each of the $23$ non-exceptional imaginary quadratic fields of class group $\cyc 2 \times \cyc 2$, we have computed the $8$ endomorphism algebras arising from restriction of scalars of Gross $\Q$-curves. The results are displayed in Table \ref{table}. The notation is as follows: for the biquadratic fields, the notation $(a,b)$ indicates the field $\Q(\sqrt{a},\sqrt{b})$; for the quaternion algebras, we write the discriminant of the algebra.

For a Gross $\Q$-curve $E$, recall that we denote by $B_E$ the abelian variety over~$\Q$ such that $\Res_{H/M}E\sim (B_E)_M$. Since $B_E$ is isogenous to its quadratic twist over~$M$, this implies that
\begin{align*}
  \Res_{H/\Q}E \sim (B_E)^2.
\end{align*}
We observe in Table \ref{table} that for all discriminants except $-195$, $-312$, $-555$, $-715$, and $-760$, at least one of the quaternion algebras is the split algebra $\M_2(\Q)$ of discriminant $1$. This implies that for the corresponding Gross $\Q$-curve $E$ the variety~$B_E$ decomposes as
\begin{align*}
  B_E\sim A^2,
\end{align*}
with $A/\Q$ an abelian surface. Therefore, $\Res_{H/\Q}E$ decomposes as the fourth power of an abelian surface.

On the other hand, for the discriminants $-195$, $-312$, $-555$, $-715$, and $-760$ we see that $B_E$ is always simple: its endomorphism algebra is either a biquadratic field or a quaternion division algebra over $\Q$. Therefore, $\Res_{H/\Q}E\sim W^2$ for some simple variety $W$ of dimension $4$. We record these findings in the following statement.
\begin{theorem}\label{th: nakamura table}
Let $M$ be an imaginary quadratic field of discriminant $D$ and Hilbert class field $H$. Suppose that $D$ is non-exceptional and that $\Gal(H/M)\simeq \cyc 2\times \cyc 2$. If $D\neq -195, -312, -555, -715, -760$, there exists a Gross $\Q$-curve $E/H$ such that 
  \begin{align*}
    \Res_{H/\Q}E \sim A^4, \text{ for some simple abelian surface } A/\Q.
  \end{align*}
If $D=-195, -312,-555,-715, -760$, then for every Gross $\Q$-curve $E/H$ we have that 
  \begin{align*}
    \Res_{H/\Q}E \sim W^2, \text{ for some simple abelian variety } W/\Q \text{ of dimension } 4.
  \end{align*}

\end{theorem}
\begin{remark}
  As mentioned above, the cases of $D=-84$ and $D=-195$ were already computed by Nakamura (\cite[\S 6]{Nak04} ). We note what appears to be a typo in Nakamura's table in page $647$: the last biquadratic field should be $\Q(\sqrt{-14},\sqrt{42})$, instead of $\Q(\sqrt{-14},\sqrt{-42})$.
\end{remark}

We have used the software Sage \cite{sage} and Magma \cite{magma} to perform the computations of Table \ref{table}. The interested reader can find the code we used in \url{https://github.com/xguitart/restriction_of_scalars_of_Q_curves}.

\renewcommand{\arraystretch}{1.5}

\begin{table}
\begin{tabular}{c|c|c}
$D$ & Biquadratic fields & Quaternion Algebras \\
\hline
$-84$ & $(-14,-2),(-6,2),(-6,-42),(-14,42)$& $2,1,2,1$\\
\hline
$-120$ & $(-5,10),(5,-10),(-5,-10),(5,10)$& $1,6,3,1$\\
\hline
$-132$ & $(22,-2),(-6,-2),(6,-66),(-22,-66)$& $1,2,1,2$\\
\hline
$-168$ & $(-14,-2),(3,-21),(14,21),(-3,2)$& $2,1,1,1$\\
\hline
$-195$ & $(13,-5),(-13,-5),(-13,5),(13,5)$& $13,39,26,39$\\
\hline
$-228$ & $(-38,-2),(6,-2),(-6,-114),(38,-114)$& $2,1,2,1$\\
\hline
$-280$ & $(-10,-5),(-10,5),(10,-5),(10,5)$& $2,1,14,14$\\
\hline
$-312$ & $(13,-26),(-13,26),(-13,-26),(13,26)$& $13,39,26,39$\\
\hline
$-372$ & $(-62,31),(-6,-3),(-6,31),(-62,-3)$& $2,1,2,1$\\
\hline
$-408$ & $(-17,34),(-17,-34),(17,-34),(17,34)$& $2,1,1,1$\\
\hline
$-435$ & $(-29,-5),(-29,5),(29,-5),(29,5)$& $2,1,1,1$\\
\hline
$-483$ & $(-23,7),(23,-69),(-21,-7),(21,69)$& $2,1,1,1$\\
\hline
$-520$ & $(-13,-5),(13,-5),(-13,5),(13,5)$& $1,1,1,2$\\
\hline
$-532$ & $(-38,-19),(-14,7),(-14,-19),(-38,7)$& $1,2,1,2$\\
\hline
$-555$ & $(37,-5),(-37,-5),(-37,5),(37,5)$& $37,111,74,111$\\
\hline
$-595$ & $(-17,85),(17,-85),(-17,-85),(17,85)$& $7,1,1,14$\\
\hline
$-627$ & $(19,-11),(-19,-57),(-33,11),(33,57)$& $1,2,1,1$\\
\hline
$-708$ & $(118,-59),(-6,3),(6,-59),(-118,3)$& $1,2,1,2$\\
\hline
$-715$ & $(-13,-65),(13,-65),(-13,65),(13,65)$& $5,10,55,55$\\
\hline
$-760$ & $(-10,5),(10,-5),(-10,-5),(10,5)$& $5,95,10,95$\\
\hline
$-795$ & $(-53,-5),(53,-5),(-53,5),(53,5)$& $6,1,1,3$\\
\hline
$-1012$ & $(-46,23),(-22,-11),(-22,23),(-46,-11)$& $2,1,2,1$\\
\hline
$-1435$ & $(-41,205),(-41,-205),(41,-205),(41,205)$& $2,1,1,1$\\

\end{tabular}\caption{Endomorphism algebras of the restriction of scalars of Gross $\Q$-curves. For the biquadratic fields, the notation $(a,b)$ indicates the field $\Q(\sqrt{a},\sqrt{b})$; for the quaternion algebras, we write the discriminant of the algebra}\label{table}
\end{table}

\clearpage

\section{Proof of the main theorems}\label{section: main proofs}
We begin with a Lemma that will be used in the proof of Theorem \ref{theorem: main}.
\begin{lemma}\label{lemma: p mid D}
  Let $E$ be a Gross $\Q$-curve with CM by a field $M$ of discriminant~$D$, and suppose that $\Gal(H/M)$ is isomorphic to $ \cyc 2 \times \cyc 2$. Denote by $\gamma_E^H$ the class in $ H^2(\Gal(H/M),M^\times)$ attached to $E$, and by $c_E$ a cocycle representing $\gamma_E^H$. If $\sigma\in\Gal(H/M)$ is non-trivial, then $\pm d\cdot c_E(\sigma,\sigma)\in  (M^{\times})^2$ for some divisor $d$ of $ D$ such that $d$ is not a square in $M^\times$.
\end{lemma}

\begin{proof}
  Let $\cO_M$ denote the ring of integers of $M$. Denote by $p_1,p_2,p_3$ the primes dividing $D$. Observe that $p_i\cO_M = \fp_i^2$, with $\fp_i$ a non-principal prime ideal of $\cO_M$. It is clear that we can always find $p_i,p_j$ such that $\pm p_ip_j$ is not a square in $M^\times$, and therefore $\fp_i\fp_j$ is not principal.  Thus $\fp_i,\fp_j$ generate the class group. Therefore, we can assume that any non-trivial element of $\Gal(H/K)$ is of the form $\sigma_\fq$ for some unramified prime $\fq$ which is equivalent to either $\fp_i$, $\fp_j$ or $\fp_i\cdot\fp_j$. Here $\sigma_\fq$ stands for the Frobenius automorphism of $H/K$ at $\fq$. 

Now we argue (and use the same notation) as in \cite[Proof of Theorem 3]{Nak04}. Namely, denote by $u(\fq)$ the $\fq$-multiplication isogenies
\begin{align*}
  u(\fq): \acc{\sigma_{\fq}}E\lra E,
\end{align*}
and denote by $c$ the $2$-cocycle associated to $E$ using the system of isogenies $u(\fq)$ (together with the identity isogeny for $1\in \Gal(H/M)$). Note that $c_E$ is any cocycle representing $\gamma_E^H$, and it may be different from $c$. But in any case they are cohomologous, which in particular implies that 
\begin{align}\label{eq: c and cE}
  c(\sigma_\fq,\sigma_\fq)=b_{\fq}^2\cdot c_E(\sigma_\fq,\sigma_\fq)\text{ for some }b_{\fq}\in M^\times.
\end{align}
From display $(6)$ and the display after that of loc. cit., since the order $n$ of $\sigma_\fq$ is $2$ in our case, we see that 
\begin{align*}
  c(\sigma_\fq,\sigma_\fq)\cO_M =\fq^2. 
\end{align*}
The proof is finished by observing that $\fq^2=\alpha\cO_M$, where $\alpha\in M^\times$ is, up to an element of $(M^\times)^2$, equal to $\pm p_i$, $\pm p_j$, or $\pm p_i\cdot p_j$.
\end{proof}

\subsubsection*{Proof of  Theorem~\ref{theorem: main}}
For all the quadratic imaginary fields not listed in \eqref{equation: fields}, we have constructed in the first part of Theorem \ref{th: nakamura table} abelian surfaces defined over $\Q$ satisfying the hypothesis of the theorem. To rule out the remaining 6 fields, we proceed in the following way.

Let $M$ be one of the fields in the list \eqref{equation: fields} and suppose that an abelian surface $A$ satisfying the hypothesis of the theorem exists for $M$. Resume the notations from Section~\ref{section: crepCMsq}. As $\Gal(H/M)\simeq \cyc 2 \times \cyc 2$ and $H\subseteq K$ (by \cite[Theorem 2.14]{FG18}), the only possibilities for $\Gal(K/M)$ are $\cyc 2 \times \cyc 2$, $\dih 4$, and $\dih 6$. 

Suppose that $\Gal(K/M)$ is $\cyc 2 \times \cyc 2$. Then $K=H$ and thus $E$ is a Gross $\Q$-curve. By Proposition~\ref{proposition: McurveGrossCurve}, we have that $M$ is not exceptional and thus we cannot have $M=\Q(\sqrt{-340})$. For the other possibilities for $M$, we have seen in the second part of Theorem \ref{th: nakamura table} that $\Res_{H/\Q} E$ does not have any simple factor of dimension $2$, but this is a contradiction with the fact that $A$ should be a factor of $\Res_{H/\Q} E$ (indeed, the universal property of Weil's restriction of scalars implies that $\Hom(A,\Res_{H/\Q}E)=\Hom(A_H,E)\simeq M^2$, and thus $\Hom(A,\Res_{H/\Q}E)\neq 0$). 

Suppose that $\Gal(K/M)$ is $\dih 4$ or $\dih 6$. Resume the notations of Section~\ref{section: obstructions}. Let~$E^*$ be a Ribet $M$-curve completely defined over $H$ with CM by $M$ which we chose as in Corollary~\ref{corollary: elementsgroup} (and which exists because of Proposition~\ref{proposition: McurveGrossCurve}). Note that Hypothesis~\ref{hypothesis: forta} is satisfied. Then, by Theorem~\ref{theorem: dihedralmagic}, there is a non-trivial element $\overline \tau\in \Gal(N/M)=\Gal(H/N)$ such that 
\begin{align}\label{eq: c(tau,tau)}
 c_{E^*}^H(\overline \tau,\overline \tau)= \pm 1 .
\end{align}
If $M$ is non-exceptional, as noted in Remark~\ref{remark: ele}, we can suppose that $E^*$ is in fact a Gross $\Q$-curve. Then \eqref{eq: c(tau,tau)} is a contradiction with Lemma \ref{lemma: p mid D}.

It remains to show that \eqref{eq: c(tau,tau)} also brings a contradiction if $M=\Q(\sqrt{-340})$ is the exceptional field. Put $T=H^{\langle \bar\tau \rangle}$, the fixed field by $\bar\tau$. Observe that $M\subsetneq T\subsetneq H$. If $c_{E^*}^H(\bar\tau,\bar\tau)=1$ then by Theorem~\ref{theorem: Weilsdescend} the curve $E^*$ is isogenous to a curve defined over $T$, and this is a contradiction with the fact that $M(j_{E^*})=H$.

Suppose now that $c_{E^*}^H(\bar\tau,\bar\tau)=-1$. We will see that we can apply the above argument to an appropriate quadratic twist of $E^*$. 
\begin{claim}\label{claim}
 There exists a quadratic extension $S/H$ such that $S/M$ is Galois with $\Gal(S/M)\simeq \dih 4$ and such that $\bar \tau$ lifts to an element of order $4$ of $\Gal(S/M)$. 
\end{claim}
We now show how this claim allows us to produce the appropriate twisted curve (and we will prove the claim later on). Define $C$ to be the $S/H$ quadratic twist of~$E^*$. By \cite[Lemma 3.13]{FG18}, the curve $C$ is an $M$-curve completely defined over~$H$ and the cohomology classes of $E^*$ and $C$ are related by
\begin{align*}
  \gamma_{C}^H = \gamma_{E^*}^H\cdot \gamma_{S},
\end{align*}
where $\gamma_S\in H^2(\Gal(H/M),\{\pm 1\})$ is the cohomology class attached to the exact sequence
\begin{align}\label{eq: emb prob}
  1 \lra \Gal(S/H)\simeq \{\pm 1\}\lra \Gal(S/M)\simeq \dih 4 \lra \Gal(H/M)\lra 1.
\end{align}
If we identify $\Gal(S/M)\simeq \langle s,t|s^4,t^2,stst \rangle$, then $\Gal(S/H)$ can be identified with the subgroup generated by $s^2$ and we can assume that $\bar \tau$ lifts to $s$. Let $c_S$ be a cocycle representing $\gamma_S$. The usual construction that associates a cohomology class to \eqref{eq: emb prob} gives that $c_{S}(\bar\tau,\bar\tau)=s\cdot s$. Since $s^2$ is the non-trivial element of $\Gal(S/H)$, it corresponds to $-1$ under the isomorphism $\Gal(S/H)\simeq \{\pm 1\}$, so that $c_{S}(\bar\tau,\bar\tau)=-1$.

We conclude that $c_{C}^H(\bar\tau,\bar\tau)= c_{E^*}^H(\bar\tau,\bar\tau)c_{S}(\bar\tau,\bar\tau)= 1$, and as before this implies that $C$ can be defined over $T$, which is a contradiction.

{\bf Proof of Claim \ref{claim}.} The Hilbert class field of M is $H=\Q(i,\sqrt{5},\sqrt{17})$. If we write $H = M(\sqrt{a},\sqrt{b})$ and suppose that $\bar\tau(\sqrt{b})=\sqrt b$, it is well known (see, e.g. \cite[\S0.4]{Arne}) that the obstruction to the existence of $S$ is given by the quaternion algebra $\left( \frac{a,ab}{M}\right)$ being nonsplit. There are $3$ possibilities for $T$, namely $T = M(\sqrt{5})$, $T = M(\sqrt{17})$, or $T = M(\sqrt{5\cdot 17})$, each one giving a different obstruction. The resulting quaternion algebras giving the obstruction are
\begin{align*}
 \left( \frac{17 \cdot 5,5}{M}\right),  \left( \frac{17 \cdot 5,17}{M}\right),  \left( \frac{17,5}{M}\right)\,.
\end{align*}
Since they are all the split, the field $S$ does exist in all three cases.

\begin{remark}
As a byproduct of the above proof, we see that there do not exist abelian surfaces over $\Q$ such that $\End(A_\Qbar)\simeq \M_2(M)$ with $M$ a quadratic imaginary field with class group $\cyc 2\times \cyc 2$ and $\Gal(K/M)\simeq \dih 4$ or $\dih 6$. As shown by the table of \cite[p. 112]{Car01}, there do exist abelian surfaces over $\Q$ such that $\End(A_\Qbar)\simeq \M_2(M)$ with $M$ a quadratic imaginary field with class group $\cyc 2$ and $\Gal(K/M)\simeq \dih 4$ (resp. $\dih 6$). If $M$ is not exceptional, Theorem~\ref{th: obstructions} and Lemma~\ref{lemma: p mid D} imply that $2$ (resp. $3$) divide the discriminant of $M$ is a necessary condition for the existence of such an $A$. The examples of the table of \cite[p. 112]{Car01} show that this is actually a necessary and sufficient condition. 
\end{remark}

\subsubsection*{Proof of Corollary~\ref{corollary: main}}

Suppose that $A$ is an abelian surface defined over $\Q$ such that $A_\Qbar\sim E \times E'$, where $E$ and $E'$ are elliptic curves defined over $\Qbar$. If $E$ and $E'$ are not isogenous, then $\End(A_\Qbar)$ is 
$$
\Q\times \Q\,, \qquad M\times \Q\,\qquad \text{or}\qquad  M_1\times M_2\,,
$$ 
where $M$, $M_1\not \simeq M_2$ are quadratic imaginary fields, depending on whether none of $E$ and $E'$ has CM, only one of $E$ and $E'$ has CM, or both of $E$ and $E'$ have CM.  In any case, note that by \cite[Proposition 4.5]{FKRS12}, both $E$ and $E'$ can be defined over $\Q$, whereby the class number of $M$, $M_1$, and $M_2$ must be~$1$. Recalling that there are $9$ quadratic imaginary fields of class number $1$, this accounts for $46$ distinct $\Qbar$-endomorphism algebras. 

If $E$ and $E'$ are isogenous, we have that $\End(A_\Qbar)$ is $\M_2(M)$ or $\M_2(\Q)$, where~$M$ is a quadratic imaginary field, depending on whether $E$ has CM or not. Assume that we are in the former case. By Theorem~\ref{theorem: FG18}, we have that $M$ has class group 1, $\cyc 2$, or $\cyc 2 \times \cyc 2$. As explained in \cite[Remark 2.20]{FG18}, for all fields~$M$ with class group $1$ (resp. $\cyc 2$), abelian surfaces $A$ over $\Q$ with $\End(A_\Qbar)\simeq \M_2(M)$ can be easily found. Indeed, let $E$ be an elliptic curve with CM by the maximal order of~$M$ and defined over $\Q$ (resp. $\Q(j_E)$). Then consider the square (resp. the restriction of scalars from $\Q(j_E)$ down to $\Q$) of $E$. If $M$ has class group $\cyc 2\times \cyc 2$, invoke Theorem~\ref{theorem: main} to obtain 18 possibilities for $M$. Taking into account that there are 18 quadratic imaginary fields of class group $\cyc 2$ (see \cite{Wat03} for example), we obtain $46$ possibilities for the endomorphism algebra of a geometrically split abelian surface over $\Q$ with $\Qbar$-isogenous factors.

\subsubsection*{An open problem} 

We wish to conclude the article with an open question.

\begin{question}
Which is the subset of $\cA$ made of the $\Qbar$-endomorphism algebras $\End(\mathrm{Jac}(C)_\Qbar)$ of geometrically split Jacobians of genus 2 curves $C$ defined over~$\Q$?
\end{question}

Again the most intriguing case is to determine how many of the 45 possibilities for $\M_2(M)$, with $M$ a quadratic imaginary field, allowed by Theorem~\ref{theorem: main} for geometrically split abelian surfaces defined over~$\Q$ still occur among geometrically split Jacobians of genus $2$ curves $C$ defined over $\Q$.
Looking at the more restrictive setting that requires $\Jac(C)$ to be \emph{isomorphic} to the square of an elliptic curve with CM by the \emph{maximal order} of $M$, G\'elin, Howe, and Ritzenthaler \cite{GHR} have shown that there are $13$ possibilities for such an $M$ (all with class number $\leq 2$).

\bibliographystyle{halpha}
\bibliography{refs}

\newcommand{\etalchar}[1]{$^{#1}$}
\def\cprime{$'$}
\begin{thebibliography}{BFGR06}

\bibitem[BCP97]{magma}
W.~Bosma, J.~Cannon, and C.~Playoust.
\newblock The {M}agma algebra system. {I}. {T}he user language.
\newblock {\em J. Symbolic Comput.}, 24(3-4):235--265, 1997.
\newblock Computational algebra and number theory (London, 1993).

\bibitem[BFGR06]{BFGR06}
Nils Bruin, E.~Victor Flynn, Josep Gonz\'alez, and Victor Rotger.
\newblock On finiteness conjectures for endomorphism algebras of abelian
  surfaces.
\newblock {\em Math. Proc. Cambridge Philos. Soc.}, 141(3):383--408, 2006.

\bibitem[BS17]{BS17}
Gaetan Bisson and Marco Streng.
\newblock On polarised class groups of orders in quartic {CM}-fields.
\newblock {\em Math. Res. Lett.}, 24(2):247--270, 2017.

\bibitem[Car01]{Car01}
Gabriel Cardona.
\newblock {\em Models racionals de corbes de gènere 2}.
\newblock PhD thesis, Universitat Politècnica de Catalunya, 2001.

\bibitem[Cre92]{Cre92}
J.~E. Cremona.
\newblock Abelian varieties with extra twist, cusp forms, and elliptic curves
  over imaginary quadratic fields.
\newblock {\em J. London Math. Soc. (2)}, 45(3):404--416, 1992.

\bibitem[FG18]{FG18}
Francesc Fit\'e and Xavier Guitart.
\newblock Fields of definition of elliptic {$k$}-curves and the realizability
  of all genus 2 {S}ato--{T}ate groups over a number field.
\newblock {\em Trans. Amer. Math. Soc.}, 370(7):4623--4659, 2018.

\bibitem[FG19]{FG19}
Francesc Fité and Xavier Guitart.
\newblock Tate module tensor decompositions and the {S}ato-{T}ate conjecture
  for certain abelian varieties potentially of $\mathrm{GL}_2$-type, 2019,
  Preprint, arXiv: 1909.11712.

\bibitem[FKRS12]{FKRS12}
Francesc Fit{\'e}, Kiran~S. Kedlaya, V{\'{\i}}ctor Rotger, and Andrew~V.
  Sutherland.
\newblock Sato-{T}ate distributions and {G}alois endomorphism modules in genus
  2.
\newblock {\em Compos. Math.}, 148(5):1390--1442, 2012.

\bibitem[FS14]{FS14}
Francesc Fit\'e and Andrew~V. Sutherland.
\newblock Sato-{T}ate distributions of twists of {$y^2=x^5-x$} and
  {$y^2=x^6+1$}.
\newblock {\em Algebra Number Theory}, 8(3):543--585, 2014.

\bibitem[GHR19]{GHR}
Alexandre Gélin, Everett Howe, and Christophe Ritzenthaler.
\newblock Principally polarized squares of elliptic curves with field of moduli
  equal to $\mathbb{Q}$.
\newblock {\em The Open Book Series}, 2(1):257–274, Jan 2019.

\bibitem[Gon11]{Gon11}
Josep Gonz\'{a}lez.
\newblock Finiteness of endomorphism algebras of {CM} modular abelian
  varieties.
\newblock {\em Rev. Mat. Iberoam.}, 27(3):733--750, 2011.

\bibitem[Gro80]{Gro80}
Benedict~H. Gross.
\newblock {\em Arithmetic on elliptic curves with complex multiplication},
  volume 776 of {\em Lecture Notes in Mathematics}.
\newblock Springer, Berlin, 1980.
\newblock With an appendix by B. Mazur.

\bibitem[Kan11]{Kan11}
Ernst Kani.
\newblock Products of {CM} elliptic curves.
\newblock {\em Collect. Math.}, 62(3):297--339, 2011.

\bibitem[KS23]{KS23}
P{\i}nar K{\i}l{\i}\c{c}er and Marco Streng.
\newblock The {CM} class number one problem for curves of genus 2.
\newblock {\em Res. Number Theory}, 9(1):Paper No. 15, 29, 2023.

\bibitem[Led01]{Arne}
Arne Ledet.
\newblock Embedding problems and equivalence of quadratic forms.
\newblock {\em Math. Scand.}, 88(2):279--302, 2001.

\bibitem[MU01]{MU01}
Naoki Murabayashi and Atsuki Umegaki.
\newblock Determination of all {${\bf Q}$}-rational {CM}-points in the moduli
  space of principally polarized abelian surfaces.
\newblock {\em J. Algebra}, 235(1):267--274, 2001.

\bibitem[Nak01]{Nak01}
Tetsuo Nakamura.
\newblock On abelian varieties associated with elliptic curves with complex
  multiplication.
\newblock {\em Acta Arith.}, 97(4):379--385, 2001.

\bibitem[Nak04]{Nak04}
Tetsuo Nakamura.
\newblock A classification of {$\bold Q$}-curves with complex multiplication.
\newblock {\em J. Math. Soc. Japan}, 56(2):635--648, 2004.

\bibitem[OS18]{OS18}
Martin Orr and Alexei~N. Skorobogatov.
\newblock Finiteness theorems for {K}3 surfaces and abelian varieties of {CM}
  type.
\newblock {\em Compos. Math.}, 154(8):1571--1592, 2018.

\bibitem[Que00]{Quer}
Jordi Quer.
\newblock {${\bf Q}$}-curves and abelian varieties of {${\rm GL}_2$}-type.
\newblock {\em Proc. London Math. Soc. (3)}, 81(2):285--317, 2000.

\bibitem[Rib92]{RiAVQ}
Kenneth~A. Ribet.
\newblock Abelian varieties over {${\bf Q}$} and modular forms.
\newblock In {\em Algebra and topology 1992 ({T}aej\u on)}, pages 53--79. Korea
  Adv. Inst. Sci. Tech., Taej\u on, 1992.

\bibitem[S{\etalchar{+}}14]{sage}
W.\thinspace{}A. Stein et~al.
\newblock {\em {S}age {M}athematics {S}oftware ({V}ersion 6.3)}.
\newblock The Sage Development Team, 2014.
\newblock {\tt http://www.sagemath.org}.

\bibitem[Sch07]{Sch07}
Matthias Sch\"{u}tt.
\newblock Fields of definition of singular {$K3$} surfaces.
\newblock {\em Commun. Number Theory Phys.}, 1(2):307--321, 2007.

\bibitem[Sha96]{Sha96}
I.~R. Shafarevich.
\newblock On the arithmetic of singular {$K3$}-surfaces.
\newblock In {\em Algebra and analysis ({K}azan, 1994)}, pages 103--108. de
  Gruyter, Berlin, 1996.

\bibitem[Shi71]{Sh71}
Goro Shimura.
\newblock On the zeta-function of an abelian variety with complex
  multiplication.
\newblock {\em Ann. of Math. (2)}, 94:504--533, 1971.

\bibitem[Sil94]{silverman}
Joseph~H. Silverman.
\newblock {\em Advanced topics in the arithmetic of elliptic curves}, volume
  151 of {\em Graduate Texts in Mathematics}.
\newblock Springer-Verlag, New York, 1994.

\bibitem[Wat04]{Wat03}
Mark Watkins.
\newblock Class numbers of imaginary quadratic fields.
\newblock {\em Math. Comp.}, 73(246):907--938, 2004.

\end{thebibliography}

\end{document}